\documentclass{amsart}
\usepackage{amssymb,mathtools,mathrsfs}
\usepackage{booktabs}
\usepackage{graphicx}
\usepackage[hidelinks]{hyperref}

\numberwithin{equation}{section}

\newtheorem{theorem}{Theorem}[section]
\newtheorem{proposition}[theorem]{Proposition}
\newtheorem{lemma}[theorem]{Lemma}

\theoremstyle{definition}
\newtheorem{definition}[theorem]{Definition}
\theoremstyle{remark}

\newcommand{\Ccal}{\mathcal C}
\newcommand{\Gcal}{\mathcal G}
\newcommand{\D}{\mathrm d}
\newcommand{\Rstar}{R_*}
\newcommand{\Fcal}{\mathcal F}

\title[Second-Level Concavity of the Riemann $\Xi$ Kernel]{Second-Level Concavity of the Riemann $\Xi$ Kernel}

\author{Michel Planat}
\author{Patrick Sol\'e}


\begin{document}

\begin{abstract}
Let $\Phi$ be the classical Jacobi-theta kernel in the Fourier representation of the Riemann $\Xi$-function, set $s(t)=\Phi(\sqrt t)$, and define the first Laguerre expression $f(t)=s'(t)^2-s(t)s''(t)$. Csordas and Dimitrov (2000) conjectured that $\log f$ is strictly concave on $(0,\infty)$; Csordas (2015) later restated the assertion as Open Problem~4.14. We prove the conjecture by two complementary methods, sharing only a local certificate near the modular fixed point. The first proof is a direct theta-series argument: a directed-rounding Taylor certificate near $t=0$ is joined to a dominant-first-summand estimate with rigorous theta-tail bounds. The second proof uses Jacobi's nonlinear third-order differential equation for $\theta_3$ to obtain a three-dimensional autonomous phase space, a sharp elliptic monotonicity theorem, and an exact quartic reduction of the target inequality. The quartic boundary is a polynomial shear of the quadratic cone $XY=Z^2$. A directed interval certificate proves that every possible cone contact on the only remaining compact interval points strictly into the desired region; beyond that interval a pointwise monotonicity theorem closes the argument. The finite certificates and exact symbolic checks are supplied as reproducible scripts. The result implies the associated double Tur\'an inequalities through the theorem of Csordas--Dimitrov, but no assertion of the Riemann Hypothesis is made.
\end{abstract}

\maketitle

\section{The conjecture and the main theorem}
Let $\Phi$ denote the Riemann theta kernel appearing in the Fourier cosine representation of the Riemann $\Xi$-function. We use Csordas' normalization
\begin{equation}\label{eq:Phi}
 \Phi(r)=\sum_{m=1}^{\infty}\pi m^2
 \bigl(2\pi m^2e^{4r}-3\bigr)
 \exp\!\bigl(5r-\pi m^2e^{4r}\bigr).
\end{equation}
It is even and positive on the real axis.  Put
\begin{equation}\label{eq:sf}
 s(t)=\Phi(\sqrt t),\qquad
 f(t)=s'(t)^2-s(t)s''(t).
\end{equation}
Theorem~4.2(b) of Csordas \cite{Csordas2015} gives the strict log-concavity of $s$, and in particular $f(t)>0$ for $t>0$; see also Coffey--Csordas \cite{CoffeyCsordas2013}.  Open Problem~4.14 of \cite{Csordas2015}, originating in Problem~3.3 of Csordas--Dimitrov \cite{CD2000}, asks whether
\begin{equation}\label{eq:CD}
 \boxed{\qquad (\log f)''(t)<0\qquad(t>0).\qquad}
\end{equation}

The goal of this article is to give two proofs of the next Theorem 
\begin{theorem}[Csordas--Dimitrov conjecture]\label{thm:main}
For the Riemann theta kernel \eqref{eq:Phi}, inequality \eqref{eq:CD} holds for every $t>0$.
\end{theorem}

The terminology ``first Laguerre expression'' in \eqref{eq:sf} is standard in this setting; it is the differential counterpart of the Tur\'an expression $\gamma_k^2-\gamma_{k-1}\gamma_{k+1}$ for coefficient sequences \cite{CravenCsordas2002,Csordas2015}.  The present theorem is a second-level kernel concavity statement, not a proof of the Riemann Hypothesis.

The proof has the following architecture.  A common local lemma proves the theorem for $0<t\le0.022$.  From the handoff $L=4\sqrt t=0.5932$ onward we give two different continuations:
\begin{enumerate}
\item[(I)]
a direct first-theta-summand plus tail proof;
\item[(II)] a Jacobi differential/elliptic phase-space proof with a cone barrier and a monotone post-barrier regime.
\end{enumerate}
The local lemma is structurally useful in Proof~II: the cone normal margin tends to zero at the modular point, so the geometric barrier is not uniform as $L\downarrow0$.

\section{Equivalent curvature quantities}
Since $f>0$, define
\begin{equation}\label{eq:Ccal}
 \Ccal(t)=f(t)f''(t)-f'(t)^2.
\end{equation}
Then
\begin{equation}\label{eq:Ceq}
 (\log f)''=\frac{\Ccal}{f^2},
\end{equation}
so \eqref{eq:CD} is equivalent to $\Ccal(t)<0$.

For Proof~I set $t=r^2$ and introduce
\begin{equation}\label{eq:ABdirect}
 A_0(r)=-\frac{\D}{\D r}\log\Phi(r),\qquad
 B_0(r)=rA_0'(r)-A_0(r).
\end{equation}
The first-level log-concavity, proved in \cite[Theorem~4.2(b)]{Csordas2015}, is equivalent to $B_0(r)>0$.  Define
\begin{equation}\label{eq:Gdirect}
 \Gcal(r)=2rB_0-r^2(\log B_0)''+r(\log B_0)'-6.
\end{equation}
A direct calculation gives
\begin{equation}\label{eq:Gidentity}
 \boxed{\qquad (\log f(t))''=-\frac{\Gcal(r)}{4r^4},\qquad t=r^2.\qquad}
\end{equation}
Thus Proof~I amounts to $\Gcal(r)>0$.

\section{Common local lemma: a certified modular neighborhood}
\begin{proposition}[Local certificate]\label{prop:local}
For $0\le t\le0.022$,
\begin{equation}\label{eq:localconcl}
 \Ccal(t)<-3613<0.
\end{equation}
Consequently Theorem~\ref{thm:main} holds on $0<t\le0.022$.
\end{proposition}

The endpoint $0.022$ is chosen deliberately: compared with the earlier cutoff $0.02$, it gives substantially larger margins at the handoff in both global continuations while leaving the Cauchy remainder comfortably below the certified Taylor margin.

\subsection{Taylor coefficients}
For $q=\pi m^2$ write
\[
 \phi_m(r)=q(2qe^{4r}-3)e^{5r-qe^{4r}}.
\]
Define integer polynomials by
\begin{equation}\label{eq:Prec}
 P_0(x)=2x-3,\qquad
 P_{n+1}(x)=(5-4x)P_n(x)+4xP_n'(x).
\end{equation}
Then
\begin{equation}\label{eq:scoef}
 [t^k]s(t)=\frac1{(2k)!}\sum_{m\ge1}q e^{-q}P_{2k}(q).
\end{equation}
The supplied directed-rounding certificate evaluates these coefficients through the order required to form the degree--$90$ Taylor polynomial $P_{90}^{\Ccal}$ of \eqref{eq:Ccal}.

\subsection{Cauchy remainder}
We use the following standard Cauchy estimate explicitly. If
\[
 h(z)=\sum_{n\ge0}a_nz^n
\]
is analytic on $|z|\le R$ and $|h(z)|\le M$ there, then Cauchy's integral formula gives $|a_n|\le MR^{-n}$. Consequently, for $0\le r<R$,
\begin{equation}\label{eq:Cauchytail}
 \left|h(z)-\sum_{n=0}^{N}a_nz^n\right|
 \le M\frac{(r/R)^{N+1}}{1-r/R},\qquad |z|\le r.
\end{equation}
We apply this elementary form with $R=0.04$, $r=0.022$, and $N=90$.

On $|t|\le0.05$, choose $r^2=t$.  The elementary estimates
\[
 |e^{4r}|<3,\qquad \Re(e^{4r})>\frac16,\qquad |e^{5r}|<4
\]
give
\[
 |\phi_m(r)|<275m^4e^{-m^2/2},\qquad |s(t)|<1100.
\]
On $|t|\le0.04$, Cauchy's estimate and the identities
\[
 f=s'^2-ss'',\quad f'=s's''-ss''',\quad f''=s''^2-ss''''
\]
yield
\begin{equation}\label{eq:MC}
 |\Ccal(t)|\le M_C:=216686800000000000000000000.
\end{equation}
Hence, for $|t|\le0.022$,
\begin{equation}\label{eq:rem022}
 |\Ccal(t)-P_{90}^{\Ccal}(t)|
 \le M_C\frac{(0.022/0.04)^{91}}{1-0.022/0.04}
 <1136.645.
\end{equation}

\subsection{Directed interval sign check}
The script \path{certificate_small_t_022.py} uses outward-rounded Decimal intervals, an interval enclosure of $\pi$, $m\le40$ in \eqref{eq:scoef}, and a symmetric $10^{-500}$ padding for the omitted theta tail.  Interval Horner evaluation on $220$ intervals of width $10^{-4}$ gives
\begin{equation}\label{eq:P90new}
 P_{90}^{\Ccal}(t)<-4750.1765\qquad(0\le t\le0.022).
\end{equation}
Combining \eqref{eq:rem022} and \eqref{eq:P90new} gives \eqref{eq:localconcl}.

The recurrence \eqref{eq:Prec} also gives a transparent omitted-$m$ bound: if $A_n$ denotes the sum of absolute coefficients of $P_n$, then
\[
 A_{n+1}\le(9+4(n+1))A_n\le765A_n\qquad(n\le188),
\]
so the $m=41$ term is already vastly below the explicit padding used in the certificate.

\section{Proof I: direct theta summand and tail}
Set
\begin{equation}\label{eq:Lydirect}
 L=4r,\qquad y=\pi e^L.
\end{equation}
Write
\begin{equation}\label{eq:splitdirect}
 \Phi(r)=\phi_1(r)(1+\varepsilon(r)),
 \qquad
 \phi_1(r)=\pi(2\pi e^{4r}-3)e^{5r-\pi e^{4r}}.
\end{equation}

\subsection{The first summand}
Let $B_1,\Gcal_1$ denote \eqref{eq:ABdirect}--\eqref{eq:Gdirect} for $\phi_1$.  Exact differentiation gives
\begin{equation}\label{eq:B1direct}
 B_1=y(4L-4)+9+\frac{12(L+1)}{2y-3}
 +\frac{36L}{(2y-3)^2}.
\end{equation}
The exact-rational Bernstein certificate \path{certificate_first_term_5932.py} proves
\begin{align}
 \Gcal_1&>0.48,
 &B_1&>2.3,
 &\frac{B_1'}{B_1}&<20,
 &\frac{B_1''}{B_1}&<260,
 &&0.5932\le L\le1.\label{eq:firststrong}\\
 \Gcal_1&>1,
 &B_1&>9,
 &\frac{B_1'}{B_1}&<13,
 &\frac{B_1''}{B_1}&<110,
 &&1\le L\le4.\label{eq:firstmid}
\end{align}
The exact one-term formula also gives $\Gcal_1>1$ for every $L\ge1$; only the displayed derivative-ratio bounds on $1\le L\le4$ are used in the tail propagation.  Here derivatives of $B_1$ are with respect to $r$.

\subsection{Tail derivatives and explicit propagation}
The relative $m$th tail summand is
\begin{equation}\label{eq:rhodirect}
 \rho_m=m^2\frac{2m^2y-3}{2y-3}e^{-(m^2-1)y},\qquad m\ge2,
\end{equation}
and $\varepsilon=\sum_{m\ge2}\rho_m$.  After removal of the exponential and the positive denominator, the numerator of $(-1)^k\rho_m^{(k)}$ has nonnegative coefficients after the shift $m^2=4+u$, $y=11/2+v$, for $0\le k\le5$.  Hence the tail derivatives alternate in sign and their absolute values decrease with $y$.

The exact rational certificate \path{certificate_tail_L.py}, evaluated at $y=5.68$, yields for derivatives with respect to $L$
\begin{equation}\label{eq:epsL}
 |\varepsilon_L'|<1.5\!\times\!10^{-5},\quad
 |\varepsilon_L''|<2.5\!\times\!10^{-4},\quad
 |\varepsilon_L^{(3)}|<0.004,\quad
 |\varepsilon_L^{(4)}|<0.055.
\end{equation}
Thus, with $r$-derivatives,
\begin{equation}\label{eq:epsr}
 |\varepsilon'|<6\!\times\!10^{-5},\quad
 |\varepsilon''|<0.004,\quad
 |\varepsilon^{(3)}|<0.26,\quad
 |\varepsilon^{(4)}|<14.1.
\end{equation}

Put
\begin{equation}\label{eq:pertvars}
 \ell=\log(1+\varepsilon),\qquad
 \zeta=\ell'-r\ell'',\qquad
 \eta=\frac{\zeta}{B_1},\qquad
 \Lambda=\log(1+\eta).
\end{equation}
Then
\begin{equation}\label{eq:DGexact}
 B_0=B_1+\zeta,
 \qquad
 \Gcal-\Gcal_1=2r\zeta-r^2\Lambda''+r\Lambda'.
\end{equation}
For auditability we display the actual propagation.  From \eqref{eq:epsr}, $r\le1/4$, and $1+\varepsilon>1$ one may take
\begin{align*}
 |\ell'|&<6.0\times10^{-5},&
 |\ell''|&<0.004000004,\\
 |\ell^{(3)}|&<0.260001,&
 |\ell^{(4)}|&<14.100111,\\
 |\zeta|&<0.001061,&
 |\zeta'|&<0.065001,&
 |\zeta''|&<3.78503.
\end{align*}
Using $B_1>2.3$, $B_1'/B_1<20$, and $B_1''/B_1<260$ gives
\[
 |\eta|<4.61\times10^{-4},\quad
 |\eta'|<0.03748,\quad
 |\eta''|<3.26463,
\]
and therefore
\[
 |\Lambda'|<0.03750,\qquad |\Lambda''|<3.26754.
\]
The exact rational propagation script \path{certificate_tail_5932.py} concludes
\begin{equation}\label{eq:DG022}
 \boxed{\quad |\Gcal-\Gcal_1|<0.22,
 \qquad0.5932\le L\le1.\quad}
\end{equation}
At $y\ge8$ the corresponding certified bound is
\begin{equation}\label{eq:DGmid2}
 |\Gcal-\Gcal_1|<0.02,\qquad1\le L\le4.
\end{equation}
For $L\ge4$ the exponential-removed derivative numerators have bounded coefficient sums and degree, yielding
\begin{equation}\label{eq:DGlarge2}
 |\Gcal-\Gcal_1|<10^{-80}.
\end{equation}

\subsection{Completion of Proof I}
The common local lemma reaches $L=0.5932$, since
\[
 (0.5932/4)^2=0.02199289<0.022.
\]
On $0.5932\le L\le1$, \eqref{eq:firststrong} and \eqref{eq:DG022} give $\Gcal>0.26$.  On $1\le L\le4$, \eqref{eq:firstmid} and \eqref{eq:DGmid2} give $\Gcal>0.98$; for $L\ge4$, \eqref{eq:DGlarge2} gives $\Gcal>1-10^{-80}$.  Equation \eqref{eq:Gidentity} completes Proof~I of Theorem~\ref{thm:main}.

\section{Proof II: Jacobi differential phase space}
Proof~II uses the special differential geometry of the Jacobi theta constant rather than direct control of $\Gcal$.

\subsection{Theta operator and Jacobi's third-order equation}
Let
\begin{equation}\label{eq:theta3}
 \theta_3(x)=\sum_{n\in\mathbb Z}e^{-\pi n^2x},\qquad x>0.
\end{equation}
We use the same real-nome convention for the companion theta constants,
\begin{equation}\label{eq:theta24}
 \theta_2(x)=\sum_{n\in\mathbb Z}e^{-\pi(n+1/2)^2x},\qquad
 \theta_4(x)=\sum_{n\in\mathbb Z}(-1)^n e^{-\pi n^2x}.
\end{equation}
With $x=e^{4r}$, direct differentiation gives
\begin{equation}\label{eq:thetaop}
 \boxed{\quad
 \Phi(r)=x^{5/4}\left(x\theta_3''(x)+\frac32\theta_3'(x)\right).
 \quad}
\end{equation}
Romik \cite[Eq.~(28)]{Romik2020} records Jacobi's nonlinear third-order equation for $\theta_3$.\par With $Y(x)=\theta_3(x)$, it reads
\begin{equation}\label{eq:JacobiODE}
\begin{aligned}
&\left(Y^2Y'''-15YY'Y''+30(Y')^3\right)^2
 +32\left(YY''-3(Y')^2\right)^3\\
&\hspace{20mm}=\pi^2Y^{10}\left(YY''-3(Y')^2\right)^2.
\end{aligned}
\end{equation}

Set
\begin{equation}\label{eq:LAa}
 L=\log x,\qquad \mathscr Y(L)=\theta_3(e^L),\qquad
 A=\pi e^L\mathscr Y^4,\qquad a=\frac{A'}A,
\end{equation}
and
\begin{equation}\label{eq:UVdef}
 U=2a'+1-a^2,\qquad
 V=a''+a^3-3aa'-a.
\end{equation}
Here and throughout Proof~II primes denote $L$-derivatives.

\begin{lemma}[Positivity before division]\label{lem:Upos}
Along the theta trajectory,
\begin{equation}\label{eq:Utheta}
 \boxed{\qquad U(L)=\pi^2e^{2L}\theta_2(e^L)^4\theta_4(e^L)^4>0.\qquad}
\end{equation}
\end{lemma}
\begin{proof}
With the standard modular parameter
\[
 m=\frac{\theta_2^4}{\theta_3^4},\qquad1-m=\frac{\theta_4^4}{\theta_3^4},
\]
the standard theta logarithmic-derivative identities give
$U=A^2m(1-m)$.  Since $A=\pi e^L\theta_3^4$, \eqref{eq:Utheta} follows.  This derivation does not use the variable $\chi$ introduced below.
\end{proof}

By \eqref{eq:Utheta} it is legitimate to set
\begin{equation}\label{eq:chi}
 \chi=\frac VU.
\end{equation}
Substitution of \eqref{eq:LAa}--\eqref{eq:UVdef} into \eqref{eq:JacobiODE} gives
\begin{equation}\label{eq:surface}
 4(V^2+U^3)=A^2U^2,
 \qquad
 A^2=4(U+\chi^2).
\end{equation}

\begin{definition}[Jacobi orbit]
The \emph{Jacobi orbit} is the particular curve
\[
 L\longmapsto(a(L),U(L),\chi(L))\in\mathbb R^3
\]
generated by $\mathscr Y(L)=\theta_3(e^L)$.  It is distinguished from an arbitrary solution of the autonomous system below.
\end{definition}

\begin{proposition}[Autonomous Jacobi system]\label{prop:flow}
The Jacobi orbit satisfies
\begin{equation}\label{eq:flow}
 \boxed{\begin{aligned}
 a'&=\tfrac12(U+a^2-1),\\
 U'&=2U(a+\chi),\\
 \chi'&=a\chi-U.
 \end{aligned}}
\end{equation}
Moreover
\begin{equation}\label{eq:init}
 a(0)=\chi(0)=0,\qquad
 U(0)=\Rstar:=\frac{\Gamma(1/4)^8}{64\pi^4}.
\end{equation}
\end{proposition}
\begin{proof}
The differential equations follow from \eqref{eq:UVdef}--\eqref{eq:surface}.  The classical Jacobi modular identities (see, e.g., \cite{DLMF20,ConwaySloane1999})
\[
 \theta_3(1/x)=\sqrt{x}\,\theta_3(x),\quad
 \theta_2(1/x)=\sqrt{x}\,\theta_4(x),\quad
 \theta_4(1/x)=\sqrt{x}\,\theta_2(x)
\]
show directly that $A$ and $U$ are even functions of $L$.  Hence $a=A'/A$ is odd.  Since $V$ is odd, $\chi=V/U$ is odd.  Thus $a(0)=\chi(0)=0$.  The classical value $\theta_3(1)=\Gamma(1/4)/(\sqrt2\,\pi^{3/4})$ gives \eqref{eq:init}.
\end{proof}

\subsection{Kernel factorization and elliptic monotonicity}
Equation \eqref{eq:thetaop} becomes
\begin{equation}\label{eq:kernelAU}
 \Phi(L/4)=\frac{A^{1/4}}{8\pi^{1/4}}
 \left[U+\frac32(a^2-1)\right].
\end{equation}

Let $K(m),E(m)$ be the complete elliptic integrals, put
\[
 K_c=K(1-m),\quad E_c=E(1-m),\quad q=m(1-m),
\]
and
\[
 X=E-(1-m)K,\qquad X_c=E_c-mK_c.
\]
The classical theta--elliptic parametrization gives
\begin{equation}\label{eq:ellipticbasic}
 e^L=\frac{K_c}{K},\qquad
 A=\frac{4KK_c}{\pi},\qquad
 m'=-Aq,
\end{equation}
and
\begin{equation}\label{eq:phaseell}
 1-a=\frac{4K_cX}{\pi},\qquad
 1+a=\frac{4KX_c}{\pi},\qquad
 U=A^2q.
\end{equation}
Define
\begin{equation}\label{eq:Rdef}
 R=\frac{U}{1-a^2}
 =\frac{KK_cq}{XX_c}.
\end{equation}

Put
\begin{equation}\label{eq:rhoell}
 \rho(m)=\frac{X(m)}{mK(m)}.
\end{equation}
Then $R=[\rho(m)\rho(1-m)]^{-1}$.  The Euler integral gives
\begin{equation}\label{eq:rhohyp}
 \rho(m)=\frac12
 \frac{{}_2F_1(\frac12,\frac12;2;m)}{{}_2F_1(\frac12,\frac12;1;m)}.
\end{equation}
The Markov representation of Dyachenko--Karp \cite[Example~8]{DyachenkoKarp2022}, specialized to these parameters, implies that $1-\rho$ is absolutely monotone.  Therefore $-\log\rho$ is absolutely monotone and $\log\rho$ is strictly concave on $(0,1)$.  (The integral representation has a mild logarithmic endpoint singularity after the usual real-axis continuation; for numerical evaluation the substitution $u=e^{-s}$ removes the misleading slow convergence at $u=0$.)

\begin{theorem}[Sharp elliptic ratio]\label{thm:Rsharp}
For $L\ge0$,
\begin{equation}\label{eq:Rsharp}
 \boxed{\qquad R(L)\ge\Rstar=\frac{\Gamma(1/4)^8}{64\pi^4},\qquad}
\end{equation}
with equality only at $L=0$, and $R'(L)>0$ for $L>0$.
\end{theorem}
\begin{proof}
Concavity of $g=\log\rho$ gives
$g(m)+g(1-m)\le2g(1/2)$.  Hence $R(m)\ge\rho(1/2)^{-2}$.  Legendre's relation and $K(1/2)=\Gamma(1/4)^2/(4\sqrt\pi)$ yield \eqref{eq:Rsharp}.  For $m<1/2$, $g'(m)>g'(1-m)$; since $m'(L)<0$, this gives $R'(L)>0$.
\end{proof}

A second ratio gives a sharper invariant region in the $(a,\kappa)$-plane.  Define
\begin{equation}\label{eq:sigmaell}
 \sigma(m)=\frac{X(m)}m
 =\int_0^{\pi/2}\frac{\cos^2\theta}{\sqrt{1-m\sin^2\theta}}\,\D\theta.
\end{equation}
H\"older's inequality shows that $\sigma$ is strictly log-convex.  Since
\begin{equation}\label{eq:RoverA}
 \frac RA=\frac{\pi}{4\sigma(m)\sigma(1-m)},
\end{equation}
the quotient $R/A$ decreases along $L>0$.  With
\begin{equation}\label{eq:kappa}
 \kappa=\frac{R'}R,
\end{equation}
we obtain
\begin{equation}\label{eq:wedge}
 \boxed{\qquad 0<\kappa<a<1,\qquad R\ge\Rstar.\qquad}
\end{equation}
The resulting curvature direction is consistent with the independent theta-function monotonicity estimates in \cite{Faulhuber2021}; that observation is not used in the proof.

\subsection{Reduced logarithmic dynamics}
From $U=R(1-a^2)$, \eqref{eq:kernelAU} gives
\begin{equation}\label{eq:Sfactor}
 \Phi(L/4)=\frac{A^{1/4}}{8\pi^{1/4}}
 (1-a^2)\left(R-\frac32\right).
\end{equation}
Let
\begin{equation}\label{eq:Gphase}
 G(L)=\log\Phi(L/4),
\end{equation}
and set
\begin{equation}\label{eq:Pphase}
 P=G'=a\left(\frac54-R\right)
 +\frac{R\kappa}{R-3/2}.
\end{equation}
The triple $(a,R,\kappa)$ satisfies the closed system
\begin{equation}\label{eq:aRk}
\boxed{\begin{aligned}
 a'&=\tfrac12(1-a^2)(R-1),\\
 R'&=R\kappa,\\
 \kappa'&=a(R+1)\kappa
 +\tfrac12\bigl[(1-a^2)R^2-2(2-a^2)R-(1+a^2)\bigr].
\end{aligned}}
\end{equation}
Let $\mathcal D$ denote this autonomous vector field and put
\begin{equation}\label{eq:Pj}
 P_j=\mathcal D^jP\qquad(j\ge1).
\end{equation}
Thus $P_1=G''$, $P_2=G^{(3)}$, etc.

The invariant inequalities \eqref{eq:wedge} give two useful signs.  First,
\begin{equation}\label{eq:Pneg}
 P<0.
\end{equation}
Indeed, with $q_0=\kappa/a\in(0,1)$,
\[
 \frac Pa=\frac54-R+\frac{Rq_0}{R-3/2}<0
\]
for $R\ge\Rstar$.  Second,
\begin{equation}\label{eq:P1neg}
 P_1=-\frac{Q}{8(2R-3)^2}<0,
\end{equation}
where
\begin{align}\label{eq:Q}
Q={}&28R^3a^2+4R^3-100R^2a^2-80R^2a\kappa+84R^2\nonumber\\
&+117Ra^2+120Ra\kappa+48R\kappa^2-165R-45a^2+45.
\end{align}
To see $Q>0$, write $Q=C_0(R)+a^2C_1(R,q_0)$.  Since $\partial_{q_0}C_1<0$ for $R\ge\Rstar$ and
\[
 C_0+C_1(R,1)=8R(4R^2-12R+15)>0,
\]
the conclusion follows.

\subsection{The clock defect and the quartic}
Because $t=L^2/16$, introduce
\begin{equation}\label{eq:uphase}
 u=\frac1L.
\end{equation}
We call $u$ the \emph{clock variable}: after the autonomous reduction \eqref{eq:aRk}, it is the only explicit occurrence of the independent variable $L$.
The first-level log-concavity is
\begin{equation}\label{eq:Hphase}
 H=uP-P_1>0.
\end{equation}
From $P,P_1<0$ it follows that $0<u<P_1/P$.  Define
\begin{equation}\label{eq:delta}
 \delta=\frac{H}{-P_1}=1-u\frac{P}{P_1},
 \qquad0<\delta<1,
\end{equation}
and
\begin{equation}\label{eq:Cab}
 C=-\frac{P^2}{P_1}>0,\qquad
 \alpha=\frac{PP_2}{P_1^2},\qquad
 \beta=\frac{P^2P_3}{P_1^3}.
\end{equation}
A direct differentiation of \eqref{eq:Hphase} gives
\[
 H'=-uH-P_2,\qquad H''=2u^2H+uP_2-P_3.
\]
Substitution into \eqref{eq:CD} gives the exact normalized reduction
\begin{equation}\label{eq:Fquartic}
 \boxed{\quad
 \Fcal(\delta;L)
 =6\delta^2(1-\delta)^2-2C\delta^3+\beta\delta-\alpha^2<0.
 \quad}
\end{equation}

\begin{lemma}[Clock-free hierarchy]\label{lem:clockfree}
The quantities $C,\alpha,\beta$ and
\begin{equation}\label{eq:gamma}
 \gamma=\frac{P^3P_4}{P_1^4}
\end{equation}
are functions of the autonomous state $(a,R,\kappa)$ alone.  The variable $\delta$ is the only quantity in \eqref{eq:Fquartic} that contains the explicit clock $u=1/L$.
\end{lemma}
\begin{proof}
Each $P_j=\mathcal D^jP$ is a rational function of $(a,R,\kappa)$.  The assertion follows from \eqref{eq:Cab}, \eqref{eq:gamma}, and \eqref{eq:delta}.
\end{proof}
This lemma legitimizes evaluating a hypothetical contact root $\delta$ with the actual clock-free coefficients at the same Jacobi state.

\subsection{Quadratic cone and inward normal}
For fixed $C$, make the triangular change
\begin{equation}\label{eq:shear}
 \widetilde\beta
 =\beta+6\delta-(12+2C)\delta^2+6\delta^3.
\end{equation}
Then
\begin{equation}\label{eq:cone}
 \boxed{\qquad \Fcal=\delta\widetilde\beta-\alpha^2.\qquad}
\end{equation}
Thus the equality surface is polynomially isomorphic to the quadratic cone $XY=Z^2$, with its ordinary $A_1$ node at the origin.  This is a moving coordinate description because $C$ varies along the orbit; only the barrier identity \eqref{eq:cone} is used below.

\begin{figure}[t]
\centering
\includegraphics[width=0.82\textwidth]{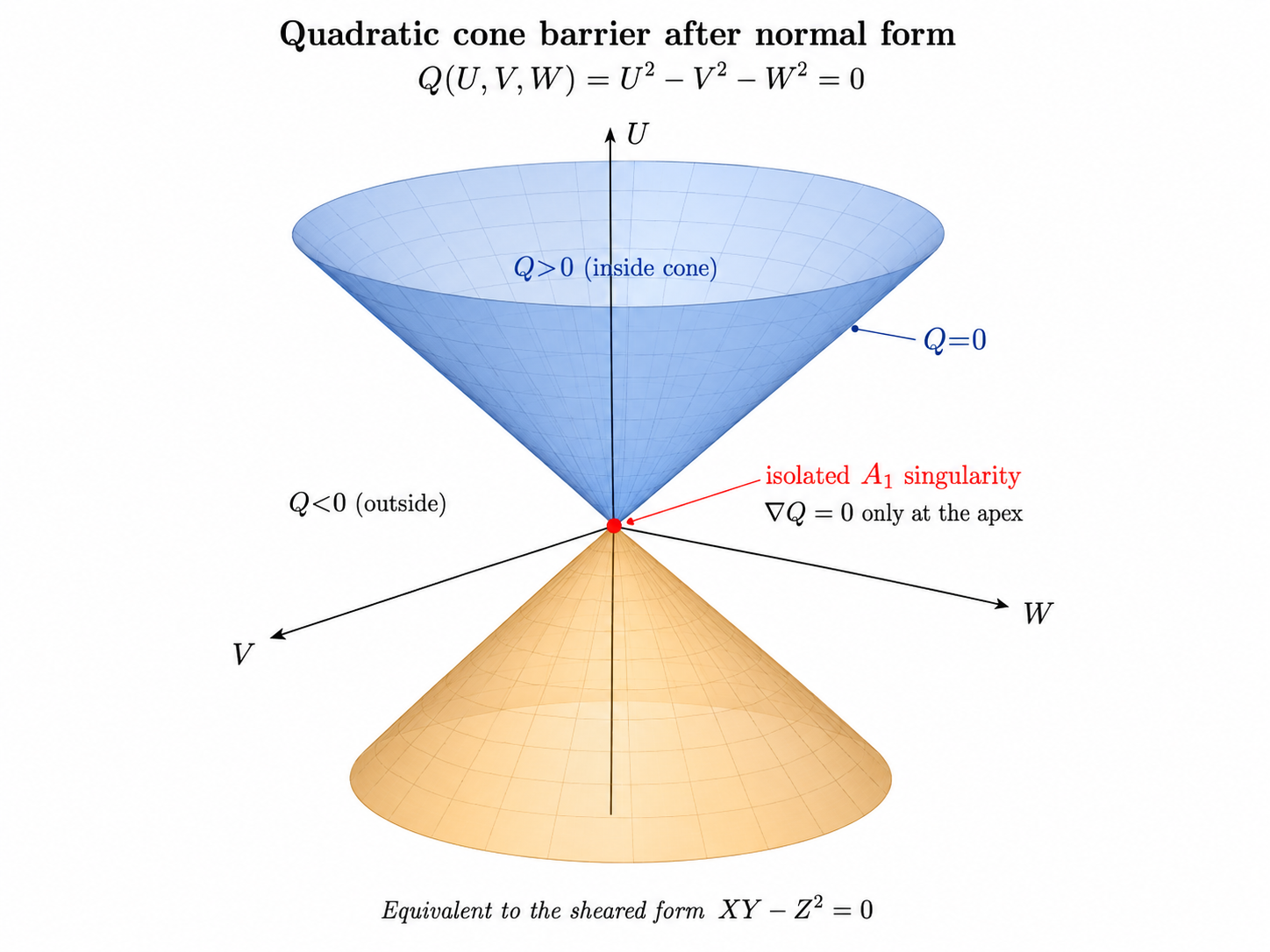}
\caption{\label{fig:cone}The static normal form of the quartic boundary in Proof~II.}
\end{figure}

Figure~\ref{fig:cone} records the static cone model behind \eqref{eq:cone}.  It should be read as a guide to the geometry, not as the full dynamical picture.  The proof uses only the barrier identity \eqref{eq:cone} and the inward-normal formula \eqref{eq:normal}; nevertheless the figure makes the role of the isolated $A_1$ singularity concrete.  The apex is the unique place where the equality surface fails to be smooth, and it is precisely near that singular tangency that the local modular certificate is needed before a global barrier argument can begin.

Put
\begin{equation}\label{eq:omega}
 \omega=-\frac PC>0.
\end{equation}
The normalized hierarchy satisfies
\begin{equation}\label{eq:hierarchy}
\begin{aligned}
 \delta'&=\omega(1-\delta)(\alpha-\delta),\\
 \alpha'&=\omega(\alpha+\beta-2\alpha^2),\\
 \beta'&=\omega(\gamma+2\beta-3\alpha\beta),\\
 C'&=\omega C(2-\alpha).
\end{aligned}
\end{equation}
At a cone contact $\Fcal=0$, set $q=\alpha/\delta$.  Eliminating $\beta$ with the cone equation gives
\begin{equation}\label{eq:normal}
 \boxed{\qquad
 \frac{\Fcal'}\omega
 =\delta\,[\gamma-\Gamma_C(\delta,q)],
 \qquad}
\end{equation}
where
\begin{align}\label{eq:GammaC}
\Gamma_C(\delta,q)=\delta\bigl[&4C\delta^2+8C\delta q-4C\delta
-18\delta^3-18\delta^2q+54\delta^2\nonumber\\
&-\delta q^2+36\delta q-54\delta
+q^3+q^2-18q+18\bigr].
\end{align}
Thus every contact is strictly inward if $\Gamma_C-\gamma>0$.

\subsection{Certified compact cone barrier}
The local Proposition~\ref{prop:local} gives $\Fcal<0$ up to the rational handoff
\begin{equation}\label{eq:L0}
 L_0=0.5932,
 \qquad (L_0/4)^2=0.02199289<0.022.
\end{equation}
We now certify that no first contact can occur before $L=0.80$.  Figure~\ref{fig:contact} summarizes the geometry of this compact step.  The green curve is the physical Jacobi trajectory $L\mapsto\delta(L)$, while the solid blue and dashed orange curves are the two real branches of the contact locus $\Fcal(\delta;L)=0$.  Their tangency at $L=0$ identifies the nonuniform modular region treated by Proposition~\ref{prop:local}; the two dotted lines mark the handoff $L_0=0.5932$ and the anchor $L=0.80$.

\begin{figure}[t]
\centering
\includegraphics[width=0.88\textwidth]{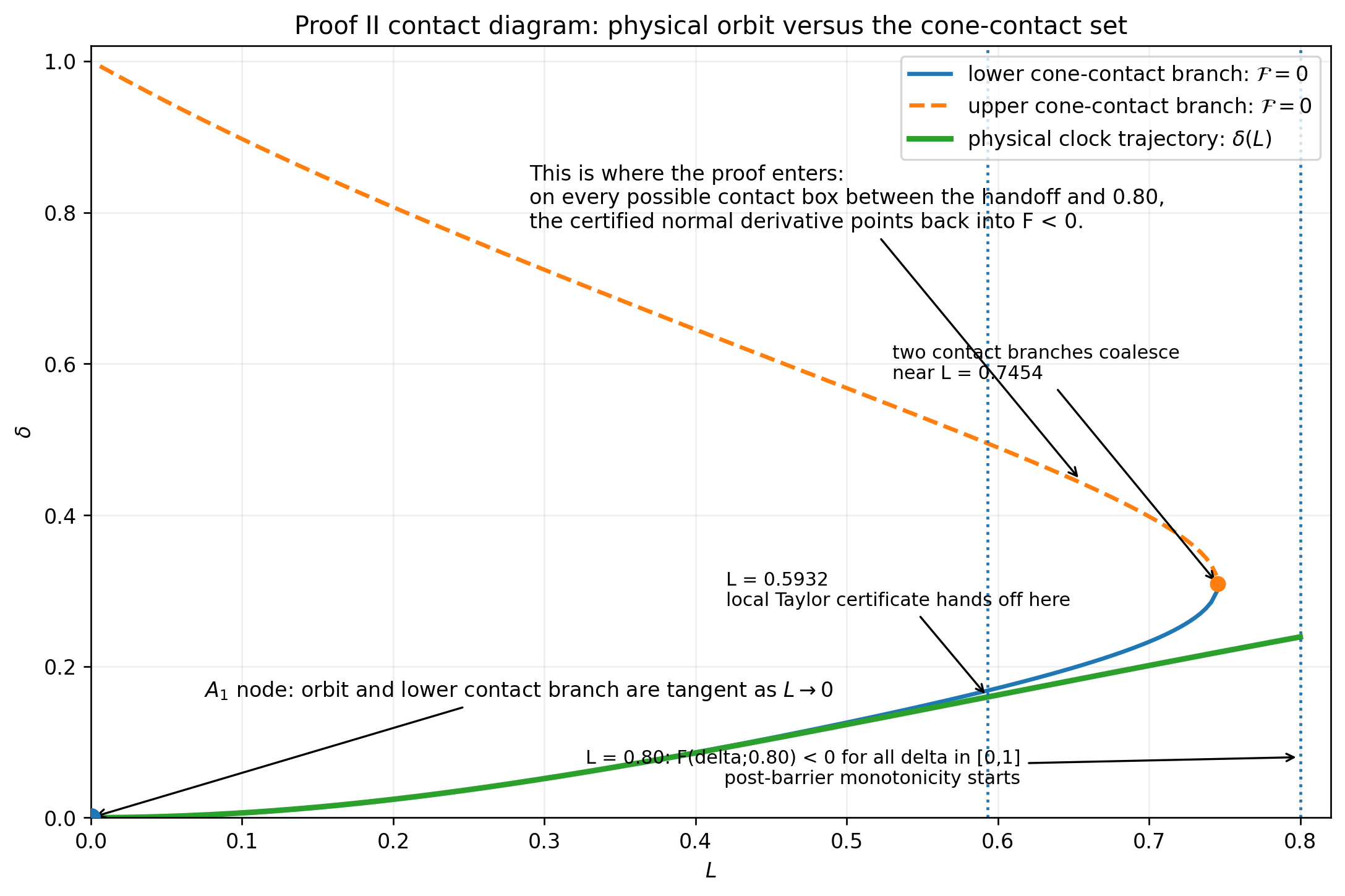}
\caption{\label{fig:contact}Contact geometry for the compact cone-barrier step in Proof~II.}
\end{figure}

Between the two dotted lines, the directed-interval certificate retains every box in which $\Fcal$ could vanish and proves $\Gamma_C-\gamma>0$ on every such box, excluding a first contact with either branch.  The visible coalescence near $L\approx0.7454$ is only a visual landmark; no fold-uniqueness or one-root argument is used.  Thus Figure~\ref{fig:contact} also explains why an all-contact certificate is preferable: the lower branch is the natural candidate for a first contact, while the upper branch makes a uniqueness-based argument unnecessarily fragile.  At $L=0.80$, \eqref{eq:F080} supplies the anchor from which the post-barrier monotonicity argument propagates $\Fcal<0$ for all larger $L$.

It is convenient to use $y=\pi e^L$.  The interval
\begin{equation}\label{eq:ycentral}
 5.68\le y\le7.00
\end{equation}
contains the image of $0.5932\le L\le0.80$.  For the first theta summand, the five derivatives $G_1',\ldots,G_1^{(5)}$ are exact rational functions of $y$.  If
\[
 \Phi=\phi_1(1+\varepsilon),\qquad \ell=\log(1+\varepsilon),
\]
the exact tail certificate \path{certificate_tail_L.py} proves, for $y\ge5.68$,
\begin{equation}\label{eq:elltailcentral}
 |\ell^{(1)}|<1.51\!\times10^{-5},\quad
 |\ell^{(2)}|<2.51\!\times10^{-4},\quad
 |\ell^{(3)}|<0.004001,
\end{equation}
\begin{equation}\label{eq:elltailcentral2}
 |\ell^{(4)}|<0.05501,\qquad
 |\ell^{(5)}|<0.701,
\end{equation}
where these derivatives are with respect to $L$.  Therefore directed intervals for the rational one-term derivatives plus \eqref{eq:elltailcentral}--\eqref{eq:elltailcentral2} give rigorous enclosures for $C,\alpha,\beta,\gamma$.

The script \path{certificate_cone_y_080.py} subdivides \eqref{eq:ycentral} into intervals of width $0.002$ and $0.05\le\delta\le1$ into intervals of width $0.001$.  The range $0\le\delta\le0.05$ is separately certified to have $\Fcal<0$.  Every box for which the interval evaluation of $\Fcal$ can contain zero is retained; on \emph{every} such box, including both positive cone branches, the normal margin satisfies
\begin{equation}\label{eq:margincone}
 \boxed{\qquad \Gamma_C-\gamma>0.18957.\qquad}
\end{equation}
The smallest certified lower bound is $0.1895705490\ldots$.  Hence \eqref{eq:normal} is strictly negative at every possible contact.  Since the physical orbit enters this interval with $\Fcal<0$, a first-contact argument proves
\begin{equation}\label{eq:Fto080}
 \Fcal(\delta(L);L)<0,
 \qquad L_0\le L\le0.80.
\end{equation}
This single certificate handles the two-root issue automatically; no fold uniqueness is required.

\subsection{The post-barrier monotonicity theorem}
We now prove a stronger statement: for fixed $\delta$, the quartic decreases with $L$ after $0.80$.
From \eqref{eq:hierarchy},
\begin{equation}\label{eq:partialF}
 \frac1\omega\frac{\partial\Fcal}{\partial L}
 =-\mathsf A\,\delta^3+\mathsf B\,\delta-\mathsf D,
\end{equation}
where
\begin{equation}\label{eq:ABD}
 \mathsf A=2C(2-\alpha),\qquad
 \mathsf B=\gamma+2\beta-3\alpha\beta,
 \qquad
 \mathsf D=2\alpha(\alpha+\beta-2\alpha^2).
\end{equation}
If $\mathsf A,\mathsf B,\mathsf D>0$, the maximum over $\delta\ge0$ of
$-\mathsf A\delta^3+\mathsf B\delta$ is
\begin{equation}\label{eq:cubicmax}
 \frac{2\mathsf B^{3/2}}{3\sqrt{3\mathsf A}}.
\end{equation}
Consequently
\begin{equation}\label{eq:Psi}
 27\mathsf A\mathsf D^2-4\mathsf B^3>0
\end{equation}
implies $\partial_L\Fcal<0$ for all $\delta\ge0$.

\begin{lemma}[Finite post-barrier interval]\label{lem:finitepost}
For $6.99\le y\le20$, the full theta kernel satisfies
\[
 \mathsf A>0,\qquad \mathsf B>0,\qquad \mathsf D>0,
 \qquad
 27\mathsf A\mathsf D^2-4\mathsf B^3>28.39.
\]
Hence $\partial_L\Fcal(\delta,L)<0$ for every $\delta\ge0$ on this range.
\end{lemma}
\begin{proof}[Directed interval certificate]
The finite-monotonicity certificate is the file\par\smallskip\noindent\path{certificate_monotone_finite_699.py}.\par\smallskip\noindent It uses the same exact one-term rational functions and the independently certified tail envelopes \eqref{eq:elltailcentral}--\eqref{eq:elltailcentral2}.  It partitions $6.99\le y\le20$ into intervals of width $0.005$.  The minimum lower enclosure for the left-hand side of \eqref{eq:Psi} is $28.39158275\ldots$.
\end{proof}

The infinite tail $y\ge20$ admits an analytic estimate.

\begin{lemma}[Analytic large-$y$ monotonicity]\label{lem:largepost}
For $y\ge20$,
\begin{equation}\label{eq:ABDcoarse}
 \mathsf A\ge\frac y2,\qquad
 \frac3y<\mathsf B<\frac6y,
 \qquad
 \mathsf D\ge\frac3y.
\end{equation}
Consequently $\partial_L\Fcal(\delta,L)<0$ for every $\delta\ge0$.
\end{lemma}
\begin{proof}
Let bars denote quantities for the first theta summand.  Exact shifted-polynomial positivity is reproduced by the file\par\smallskip\noindent\path{certificate_monotone_asymptotic.py}.\par\smallskip\noindent It proves for $y\ge20$
\begin{equation}\label{eq:firstasymp}
 \frac y2<-\overline P_j<2y\quad(0\le j\le4),
 \qquad0<\overline\alpha,\overline\beta<1,
 \qquad0<\overline C<y,
\end{equation}
and
\begin{equation}\label{eq:ABDfirst}
 \overline{\mathsf A}\ge y,
 \qquad \frac{19}{5y}<\overline{\mathsf B}<\frac5y,
 \qquad \overline{\mathsf D}\ge\frac4y.
\end{equation}

For the relative theta tail, exact alternating-sign derivative polynomials give, for $0\le k\le5$,
\begin{equation}\label{eq:epstaily20}
 |\varepsilon^{(k)}_L|
 \le31(3y)^k e^{-3y}<y^{-5}.
\end{equation}
The last inequality is strongest at $k=5,y=20$ and improves thereafter.  The derivative formulas for $\ell=\log(1+\varepsilon)$ then give
\begin{equation}\label{eq:elltaily20}
 |\ell^{(k)}|<2y^{-5},\qquad1\le k\le5.
\end{equation}

Using \eqref{eq:firstasymp} and the mean-value theorem on the rational maps
\[
 C=-P^2/P_1,\quad
 \alpha=PP_2/P_1^2,\quad
 \beta=P^2P_3/P_1^3,\quad
 \gamma=P^3P_4/P_1^4,
\]
one obtains the conservative perturbation bounds
\begin{align}\label{eq:pertconst}
 |C-\overline C|&<198y^{-5},&
 |\alpha-\overline\alpha|&<1080y^{-6},\\
 |\beta-\overline\beta|&<14580y^{-6},&
 |\gamma-\overline\gamma|&<174960y^{-6}.
\end{align}
For example, along the segment joining the one-term and full derivative vectors, $|P_j|\le3y$ and $|P_1|\ge y/3$, so the partial derivatives of $\alpha$ sum to at most $540/y$, those of $\beta$ to at most $7290/y$, and those of $\gamma$ to at most $87480/y$; multiplying by the derivative perturbation $2y^{-5}$ gives \eqref{eq:pertconst}.  These estimates imply
\[
 |\mathsf A-\overline{\mathsf A}|<3000y^{-5},\quad
 |\mathsf B-\overline{\mathsf B}|<260000y^{-6},\quad
 |\mathsf D-\overline{\mathsf D}|<125000y^{-6}.
\]
At $y=20$ these errors are already smaller than the gaps needed to pass from \eqref{eq:ABDfirst} to \eqref{eq:ABDcoarse}, and the comparison improves with $y$.  Notice that $\mathsf B$ and $\mathsf D$ are second-order cancellation quantities: their defining combinations tend to zero although the normalized ratios $\alpha,\beta,\gamma$ tend to quantities of order one.  The lower constant $19/5$ in \eqref{eq:ABDfirst} is therefore tuned to the endpoint $y=20$, where it still has a positive certified margin.

Finally, by \eqref{eq:cubicmax} and \eqref{eq:ABDcoarse},
\[
 \max_{\delta\ge0}(-\mathsf A\delta^3+\mathsf B\delta)
 \le\frac8{y^2}<\frac3y\le\mathsf D,
\]
which proves the lemma.
\end{proof}

It remains only to anchor this monotone region.  We deliberately place the anchor at $L=0.80$, safely beyond the tangency region of the quartic boundary.  The directed certificate
\path{certificate_no_root_080.py}, using
$6.9917<\pi e^{0.80}<6.9918$, proves
\begin{equation}\label{eq:F080}
 \boxed{\qquad \Fcal(\delta;0.80)<-0.04482
 \qquad(0\le\delta\le1).\qquad}
\end{equation}
The certified upper bound is $-0.0448287020\ldots$; this larger margin is useful because the earlier anchor $L=0.75$ lies very close to the quartic tangency.  Lemmas~\ref{lem:finitepost}--\ref{lem:largepost} show that, for each fixed $\delta$, $\Fcal(\delta;L)$ decreases for all $L\ge0.80$.  Hence \eqref{eq:F080} implies
\begin{equation}\label{eq:Fallpost}
 \Fcal(\delta;L)<0\qquad(0\le\delta\le1,\ L\ge0.80).
\end{equation}
The cone-barrier conclusion and Proposition~\ref{prop:local} complete Proof~II.\par Hence Theorem~\ref{thm:main} follows a second time.

\section{Comparison of the two proofs}
The two arguments are complementary rather than redundant.

Proof~I treats $\Phi$ directly as a rapidly convergent theta series.  Once the modular neighborhood is left, the first summand creates the sign and the remaining summands are an exponentially small perturbation.  Its strength is quantitative transparency and a short global continuation.

Proof~II explains the same sign through the special Jacobi differential structure.  Jacobi's third-order ODE reduces the kernel to an autonomous flow; elliptic log-concavity and log-convexity produce the sharp invariant region; the fourth-order concavity question becomes the scalar quartic \eqref{eq:Fquartic}; and its equality surface is a quadratic cone in a moving polynomial frame.  The finite certificate is used only on the compact transition interval where a cone crossing is possible.  The large-$L$ continuation is then a monotonicity theorem for the reduced quartic, not a root count.

Both proofs use Proposition~\ref{prop:local}.  This common input is natural: in the cone normalization the physical trajectory and the relevant cone branch become tangent as $L\downarrow0$, and the normal margin tends to zero.  The local modular expansion is therefore not merely a convenient numerical patch.  Figures~\ref{fig:cone} and \ref{fig:contact} make this geometric mechanism explicit: the fixed cone normal form identifies the isolated singularity, while the contact diagram shows where the all-contact certificate and the post-barrier monotonicity theorem intervene.

\section{Tur\'an consequence and scope}
Csordas and Dimitrov introduced the concavity condition \eqref{eq:CD} as a sufficient kernel condition for the double Tur\'an inequalities.  In Csordas' later formulation, with normalized coefficient sequence $\gamma_k$ and
\[
 T_k=\gamma_k^2-\gamma_{k-1}\gamma_{k+1},\qquad
 E_k=T_k^2-T_{k-1}T_{k+1},
\]
an affirmative solution of Open Problem~4.14 yields $E_k\ge0$ through the earlier Csordas--Dimitrov theorem \cite{CD2000,Csordas2015}.  This places Theorem~\ref{thm:main} in the classical Laguerre--Tur\'an hierarchy and the earlier moment-inequality program studied in \cite{CV1988,CravenCsordas2002,DimitrovLucas2011}.  It remains a necessary-condition result in the Riemann-$\Xi$ program and does not prove the Riemann Hypothesis.

\section{Certificate architecture and reproducibility}\label{sec:certificates}
The computer-assisted portions are deliberately finite and separated from the analytic reductions.  This section records the arithmetic model and the subdivision logic so that the certificates can be reproduced independently of the supplied implementation.

\subsection{Arithmetic model}
Polynomial identities and shifted-polynomial positivity tests are performed over the integers or rational numbers.  For interval calculations, each real quantity is represented by a closed interval $[x_-,x_+]$ with directed lower and upper arithmetic.  Algebraic operations use the corresponding inclusion-isotone interval extensions.  Exponential endpoints are evaluated at higher precision and then enlarged before they are inserted into the directed interval computation.  Consequently every displayed certificate is an enclosure statement: a box is discarded only when the relevant interval extension excludes the critical value.

The theta series are truncated only after an explicit tail estimate.  In the local Taylor computation the terms $m\le40$ are evaluated with directed intervals and the omitted terms are absorbed into a symmetric padding smaller than $10^{-500}$.  In the global theta-tail estimates the sign-alternating derivative polynomials are first reduced to polynomials with nonnegative coefficients after the shifts $m^2=4+u$ and $y=y_0+v$; monotonicity then reduces each infinite tail to an endpoint estimate.

\subsection{Local Taylor certificate}
The script \path{certificate_small_t_022.py} constructs the coefficients in \eqref{eq:scoef} through the order needed for $P_{90}^{\Ccal}$.  The interval $[0,0.022]$ is divided into $220$ subintervals of width $10^{-4}$, and interval Horner evaluation gives the uniform upper bound \eqref{eq:P90new}.  The analytic Cauchy estimate \eqref{eq:Cauchytail} then supplies the independent remainder bound \eqref{eq:rem022}.  Thus the finite calculation proves only the sign of a polynomial on a compact interval; the passage from the polynomial to the exact analytic function is entirely explicit in the text.

\subsection{Direct-theta continuation}
The script \path{certificate_first_term_5932.py} proves the rational Bernstein bounds \eqref{eq:firststrong}--\eqref{eq:firstmid}.  The scripts \path{certificate_tail_L.py} and \path{certificate_tail_5932.py} establish the derivative envelopes and their propagation through \eqref{eq:DGexact}.  The inequalities in Section~4.2 are displayed precisely so that the main perturbation step can be checked without executing code.  The remaining scripts \path{certificate_tail.py} and \path{certificate_large_r.py} cover the middle and super-exponentially small large-$L$ tails.

\subsection{All-contact cone certificate}
For Proof~II, \path{certificate_cone_y_080.py} works on
\[
 5.68\le y\le7.00,
 \qquad 0\le\delta\le1,
\]
which contains $0.5932\le L\le0.80$.  The $y$-interval is divided into boxes of width $0.002$ and the range $0.05\le\delta\le1$ into boxes of width $0.001$; $0\le\delta\le0.05$ is certified negative separately.  For each box the clock-free quantities $C,\alpha,\beta,\gamma$ are enclosed from the one-term rational functions plus the certified theta-tail envelopes.  A box is retained if and only if the interval extension of $\Fcal$ contains zero.  On every retained box, including boxes meeting either positive cone branch, the interval extension of $\Gamma_C-\gamma$ has positive lower endpoint.  The smallest lower endpoint is
\[
 0.1895705490382078775\ldots,
\]
which proves the all-contact implication used in \eqref{eq:normal}.  No numerical root selection or fold-uniqueness hypothesis enters this certificate.

\subsection{Post-barrier certificates}
At the safer anchor $L=0.80$, \path{certificate_no_root_080.py} subdivides $0\le\delta\le1$ into intervals of width $2\times10^{-4}$ and proves
\[
 \Fcal(\delta;0.80)<-0.0448287020\ldots.
\]
For the finite monotonicity range, \path{certificate_monotone_finite_699.py} divides $6.99\le y\le20$ into intervals of width $0.005$.  On every interval it proves
\[
 \mathsf A>0,\qquad \mathsf B>0,\qquad \mathsf D>0,
 \qquad 27\mathsf A\mathsf D^2-4\mathsf B^3>0,
\]
with certified minimum lower enclosure $28.39158275\ldots$.  The large-$y$ continuation is not an interval sweep: \path{certificate_monotone_asymptotic.py} verifies the exact shifted-polynomial positivity underlying \eqref{eq:firstasymp}--\eqref{eq:ABDfirst}, after which the analytic estimates \eqref{eq:epstaily20}--\eqref{eq:ABDcoarse} close the proof for every $y\ge20$.

\subsection{Minimal independent verification checklist}
A verifier does not need to reproduce every diagnostic quantity printed by the scripts.  The proof depends only on the following finite statements.
\begin{enumerate}
\item The interval polynomial bound $P_{90}^{\Ccal}<-4750.1765$ on $[0,0.022]$, together with the analytic Cauchy remainder $<1136.645$.
\item The one-term inequalities in \eqref{eq:firststrong}--\eqref{eq:firstmid} and the four tail-derivative bounds in \eqref{eq:epsL}; the displayed propagation then yields \eqref{eq:DG022} without any further opaque computation.
\item On every box in $5.68\le y\le7.00$, $0\le\delta\le1$ for which the interval extension of $\Fcal$ contains zero, the interval extension of $\Gamma_C-\gamma$ has positive lower endpoint.  The reported global lower bound $0.1895705490\ldots$ is stronger than what the first-contact argument requires.
\item At $L=0.80$, the interval extension of $\Fcal$ is negative on every $\delta$-box; the maximal certified upper endpoint is $-0.0448287020\ldots$.
\item On $6.99\le y\le20$, the interval extensions of $\mathsf A,\mathsf B,\mathsf D$ are positive and that of $27\mathsf A\mathsf D^2-4\mathsf B^3$ has lower endpoint $>28.39$.  Beyond $y=20$, only the exact shifted-polynomial checks entering \eqref{eq:firstasymp}--\eqref{eq:ABDfirst} are computational; the remainder of Lemma~\ref{lem:largepost} is analytic.
\end{enumerate}
These five checks are logically sufficient for the computer-assisted parts of the manuscript.  In particular, no floating-point root finder, graphical observation, or unverified asymptotic extrapolation is used in either proof.

\subsection{Supplementary files}
The proof bundle contains the scripts just listed, their recorded outputs, the shared interval helper modules, a one-command runner, and SHA-256 checksums.  It requires only Python and SymPy and makes no network calls.  All numerical constants in the manuscript are outputs of these scripts or deliberately weakened roundings of them.  An independent Arb/FLINT reimplementation would be a valuable external audit, but no such reimplementation is assumed in either proof.

\appendix
\section{Algebra behind the normalized cone}\label{app:conealgebra}
For completeness, we record the two eliminations that turn the original second-level concavity inequality into the scalar quartic and then into the cone-normal condition.  These are exact algebraic identities; no interval arithmetic enters this appendix.

\subsection{From the logarithmic curvature to the quartic}
Let $G(L)=\log\Phi(L/4)$, $P=G'$, and $P_j=G^{(j+1)}$ as in \eqref{eq:Pj}.  With $u=L^{-1}$ and $H=uP-P_1$, direct differentiation gives
\begin{equation}\label{eq:appHder}
 H'=-uH-P_2,
 \qquad
 H''=2u^2H+uP_2-P_3.
\end{equation}
The change of variables $t=L^2/16$ gives $\frac{\D}{\D t}=8u\frac{\D}{\D L}$.  Substitution of \eqref{eq:appHder} into the second derivative of the logarithm of the first Laguerre expression yields the equivalent inequality
\begin{equation}\label{eq:appCDred}
 6u^2-2H-\frac{P_3}{H}-\left(\frac{P_2}{H}\right)^2<0.
\end{equation}
This is the point at which all derivatives have been eliminated in favor of the autonomous hierarchy and the single explicit variable $u$.

Now use
\[
 H=-P_1\delta,
 \qquad
 u=(1-\delta)\frac{P_1}{P},
\]
which follows directly from \eqref{eq:delta}.  Multiply \eqref{eq:appCDred} by the positive factor
\[
 \frac{H^2P^2}{P_1^4}>0.
\]
Then insert the definitions
\[
 C=-\frac{P^2}{P_1},\qquad
 \alpha=\frac{PP_2}{P_1^2},\qquad
 \beta=\frac{P^2P_3}{P_1^3}
\]
gives, without approximation,
\[
 6\delta^2(1-\delta)^2-2C\delta^3+\beta\delta-\alpha^2<0,
\]
which is \eqref{eq:Fquartic}.  This calculation also explains why $\delta$ is the only coordinate retaining the explicit clock: $C,\alpha,\beta$ are formed solely from the autonomous derivatives $P_j$.

\subsection{The cone shear and the normal derivative}
Expanding the quartic and grouping the terms linear in $\beta$ gives
\[
 \Fcal
 =\delta\left[\beta+6\delta-(12+2C)\delta^2+6\delta^3\right]-\alpha^2.
\]
Thus the triangular change \eqref{eq:shear} gives the exact cone equation
\[
 \Fcal=\delta\widetilde\beta-\alpha^2.
\]
For fixed $C$ this is the standard rank-three quadratic cone $XY=Z^2$.  The transformation is used only as a local moving frame along the Jacobi orbit; because $C=C(L)$, no fixed ambient cone is being asserted globally.

To obtain the first-contact derivative, differentiate $\Fcal$ using the hierarchy \eqref{eq:hierarchy}.  At a contact point $\Fcal=0$ with $\delta>0$, put $q=\alpha/\delta$.  The contact equation eliminates $\beta$ in the form
\begin{equation}\label{eq:appbeta}
 \beta=q^2\delta-6\delta(1-\delta)^2+2C\delta^2.
\end{equation}
Substituting \eqref{eq:appbeta} into $\Fcal'/\omega$ and collecting powers of $\delta$ and $q$ gives
\[
 \frac{\Fcal'}{\omega}
 =\delta\,[\gamma-\Gamma_C(\delta,q)],
\]
with $\Gamma_C$ exactly as in \eqref{eq:GammaC}.  Therefore the all-contact certificate in Section~7.4 is a genuine normal-vector test: it proves that every point of the equality boundary that can be reached on the compact transition interval has strictly negative orbital derivative.  This is the only geometric fact about the cone needed in Proof~II.

\section*{Declaration on the use of artificial intelligence}
Generative-AI tools, principally OpenAI ChatGPT and Anthropic Claude, were used during the exploratory development of some algebraic reductions, for symbolic and numerical checking, assistance in developing verification code, and manuscript editing. All mathematical statements, proofs, computations, references, and accompanying code were reviewed by the authors, who take full responsibility for their correctness. All computer-assisted steps used in the proofs are supplied in reproducible form and do not rely on unverifiable AI output.


\begin{thebibliography}{99}

\bibitem{CD2000}
G.~Csordas and D.~K. Dimitrov,
\emph{Conjectures and theorems in the theory of entire functions},
Numer. Algorithms \textbf{25} (2000), no.~1--4, 109--122.
DOI: 10.1023/A:1016604906346.

\bibitem{Csordas2015}
G.~Csordas,
\emph{Fourier transforms of positive definite kernels and the Riemann $\xi$-function},
Comput. Methods Funct. Theory \textbf{15} (2015), no.~3, 373--391.
DOI: 10.1007/s40315-014-0105-8.

\bibitem{CV1988}
G.~Csordas and R.~S. Varga,
\emph{Moment inequalities and the Riemann hypothesis},
Constr. Approx. \textbf{4} (1988), no.~1, 175--198.
DOI: 10.1007/BF02075457.

\bibitem{CravenCsordas2002}
T.~Craven and G.~Csordas,
\emph{Iterated Laguerre and Tur\'an inequalities},
J. Inequal. Pure Appl. Math. \textbf{3} (2002), no.~3, Article~39, 14~pp.

\bibitem{CoffeyCsordas2013}
M.~W. Coffey and G.~Csordas,
\emph{On the log-concavity of a Jacobi theta function},
Math. Comp. \textbf{82} (2013), no.~284, 2265--2272.
DOI: 10.1090/S0025-5718-2013-02681-6.

\bibitem{DimitrovLucas2011}
D.~K. Dimitrov and F.~R. Lucas,
\emph{Higher order Tur\'an inequalities for the Riemann $\xi$-function},
Proc. Amer. Math. Soc. \textbf{139} (2011), no.~3, 1013--1022.
DOI: 10.1090/S0002-9939-2010-10515-4.

\bibitem{DLMF20}
NIST Digital Library of Mathematical Functions,
\emph{Chapter 20: Theta Functions}, \S20.7(viii), ``Transformations of Lattice Parameter,''
National Institute of Standards and Technology.

\bibitem{ConwaySloane1999}
J.~H. Conway and N.~J.~A. Sloane,
\emph{Sphere Packings, Lattices and Groups},
3rd ed., Grundlehren der Mathematischen Wissenschaften, vol.~290,
Springer-Verlag, New York, 1999.
DOI: 10.1007/978-1-4757-6568-7.

\bibitem{Romik2020}
D.~Romik,
\emph{The Taylor coefficients of the Jacobi theta constant $\theta_3$},
Ramanujan J. \textbf{52} (2020), no.~2, 275--290.
DOI: 10.1007/s11139-018-0109-5.

\bibitem{DyachenkoKarp2022}
A.~Dyachenko and D.~Karp,
\emph{Integral representations of ratios of the Gauss hypergeometric functions with parameters shifted by integers},
Mathematics \textbf{10} (2022), no.~20, Article~3903, 26~pp.
DOI: 10.3390/math10203903.

\bibitem{Faulhuber2021}
M.~Faulhuber,
\emph{Extremal determinants of Laplace--Beltrami operators for rectangular tori},
Math. Z. \textbf{297} (2021), 175--195.
DOI: 10.1007/s00209-020-02507-7.

\end{thebibliography}
\end{document}